\documentclass[reqno,11pt]{amsart}
\usepackage{amsthm}

\newtheorem{thm}{Theorem}[section]
 \newtheorem{cor}[thm]{Corollary}
 \newtheorem{lem}[thm]{Lemma}
 \newtheorem{prop}[thm]{Proposition}
 \theoremstyle{definition}
 \newtheorem{defn}[thm]{Definition}
 \theoremstyle{remark}
 \newtheorem{rem}[thm]{Remark}
 
 \numberwithin{equation}{section}

\usepackage{mathrsfs}
\newcommand{\V}{\mathcal V}

\begin{document}

\title[Generalized quaternionic Schur functions]
{Generalized quaternionic Schur functions in the ball and half-space and Krein-Langer factorization}

\author[D. Alpay]{Daniel Alpay}
\address{(DA) Department of Mathematics\\
Ben-Gurion University of the Negev\\
Beer-Sheva 84105 Israel} \email{dany@math.bgu.ac.il}
\author[F. Colombo]{Fabrizio Colombo}
\address{(FC) Politecnico di
Milano\\Dipartimento di Matematica\\ Via E. Bonardi, 9\\20133
Milano, Italy}
\email{fabrizio.colombo@polimi.it}
\author[I. Sabadini]{Irene Sabadini}
\address{(IS) Politecnico di
Milano\\Dipartimento di Matematica\\ Via E. Bonardi, 9\\20133
Milano, Italy}
\email{irene.sabadini@polimi.it}

\thanks{D. Alpay thanks the
Earl Katz family for endowing the chair which supported his
research, and the Binational Science Foundation Grant number
2010117. I. Sabadini was partially supported by the FIRB Project {\em Geometria Differenziale Complessa e Dinamica Olomorfa}}


\begin{abstract}
In this paper we prove a new version of Krein-Langer factorization theorem in the slice hyperholomorphic setting which is more general than the one proved in \cite{acs3}. We treat both the case of functions with $\kappa$ negative squares defined on subsets of the quaternionic unit ball or on subsets of the half space of quaternions with positive real part. A crucial tool in the proof of our results is the Schauder-Tychonoff theorem and an invariant subspace theorem for contractions in a Pontryagin space.
\end{abstract}

\maketitle

\section{Introduction}
\setcounter{equation}{0}
\subsection{Some history}
Functions analytic and contractive in the open unit disk $\mathbb
D$ play an important role in various fields of mathematics, in
electrical engineering and digital signal processing.
They bear various names, and in particular are called Schur functions.
We refer to \cite{hspnw} for reprints of some of the original works.
Andr\'e Bloch's 1926 memoir \cite{bloch-1926} contains also valuable historical
background.
\smallskip

Schur functions admit a number of generalizations,  within function
theory of one complex variable and outside. A $\mathbb C^{r\times
s}$-valued function $S$ analytic in $\mathbb D$ is a Schur
function if and only if the kernel
\[
K_S(z,w)=\frac{I_r-S(z)S(w)^*}{1-z\bar w}
\]
is positive definite in $\mathbb D$. In fact, much more is true.
It is enough to assume that the kernel $K_S(z,w)$ is positive
definite in some subset $\Omega$ of $\mathbb D$ to insure that
$S$ is the restriction to $\Omega$ of a (not necessarily unique)
function analytic and contractive in $\mathbb D$; see
\cite{MR2002b:47144,donoghue}.\smallskip

A Schur function has no poles inside the open unit disk. Motivated
by lowering a lower bound given by Carath\'eodory and Fej\'er in
an interpolation problem, Takagi considered in
\cite{takagi,takagi2} rational functions bounded by $1$ in
modulus on the unit circle and with poles inside $\mathbb D$.
These are the first instances of what is called a generalized
Schur function. Later studies of such functions include Chamfy
\cite{chamfy}, Dufresnoy \cite{MR20:5861} (these authors being
motivated by the study of Pisot numbers), Delsarte, Genin and
Kamp \cite{bulbs,DDGK} and Krein and Langer \cite{kl1,kl2,kl3}, to
mention a few names. The precise definition of a generalized
Schur function was given (in the setting of operator valued
functions) by Krein and Langer \cite{kl1}:

\begin{defn}
\label{negsquares} A $\mathbb C^{r\times s}$-valued function
analytic in an open subset $\Omega$ of the unit disc is called a
generalized Schur function if the kernel $K_S$ has a finite
number  (say $\kappa$) of negative  squares in $\Omega$, meaning
that for every choice of $N\in\mathbb N$, $c_1,\ldots,
c_N\in\mathbb C^r$ and $w_1,\ldots ,w_N\in\Omega$, the $N\times
N$ Hermitian matrix with $(\ell,j)$ entry
$c_\ell^*K_S(w_\ell,w_j)c_j$ has at most $\kappa$ strictly
negative eigenvalues, and exactly  $\kappa$ strictly negative
eigenvalues for some choice of $N,c_1,\ldots, c_N,w_1,\ldots
,w_N$.
\end{defn}

In the setting of matrix-valued functions, the result of Krein
and Langer states that $S$ is a generalized Schur function if and
only if it is the restriction to $\Omega$ of a function of the
form
\[
B_0(z)^{-1}S_0(z),
\]
where $S_0$ is a $\mathbb C^{r\times s}$-valued Schur function
and $B_0$ is $\mathbb C^{r\times r}$-valued Blaschke product of
degree $\kappa$. Besides \cite{kl1}, there exist various proofs
of this result; see for instance \cite[p. 141]{adrs},
\cite{MR2002664}.\smallskip

That one cannot remove the analyticity condition in the result of
Krein and Langer when $\kappa>0$ is illustrated by the well known
counterexample $S(z)=\delta_0(z)$, where $\delta_0(z)=0$ if $z\not=0$ and $\delta_0(0)=1$ (see for instance \cite[p. 82]{adrs}).
Taking into account that $z^n\delta_0(z)\equiv 0$ for $n>0$ we
have
\[
\begin{split}
\frac{1-S(z)S(w)^*}{1-z\bar w}&=\frac{1}{1-z\bar w}-\sum_{n=0}^\infty
z^n\delta_0(z)\bar w^{n}\delta_0(w)\\
&=\frac{1}{1-z\bar w}- \delta_0(z)\delta_0(w).
\end{split}
\]
The reproducing kernel Hilbert space associated with $\delta_0(z)\delta_0(w)$ is
$\mathbb C\delta_0$ and has a zero intersection with $\mathbf
H_2(\mathbb D)$, and hence the kernel $K_S$ has one negative square.

\subsection{The slice hyperholomorphic case}
Schur functions have been extended to numerous settings, and we
mention in particular the setting of  several complex variables
\cite{agler-hellinger,btv}, compact Riemann surfaces \cite{av3}
and hypercomplex functions \cite{alss1,assv}.  Generalized Schur
functions do not exist necessarily in all these
settings.\smallskip

In \cite{acs1} we began a study of Schur analysis in the framework
of slice hyperholomorphic functions. The purpose of this paper
is to prove the theorem of Krein and Langer (we
considered a particular case in \cite{acs3}) and we treat both the
unit ball and half-space cases in the quaternionic setting. To that purpose we need in
particular the following:
\begin{itemize}
\item [(i)] The notion of negative squares and of reproducing kernel
Pontryagin spaces in the quaternionic setting. This was done in
\cite{as3}.

\item[(ii)] The notion of generalized Schur functions and of
Blaschke products, see \cite{MR3127378}.

\item [(iii)] A result on invariant subspaces of contractions in
quaternionic Pontryagin spaces.

\item[(iv)] The notion of realization in the
slice-hyperholomorphic setting, in particular when the state
space is a one-sided (as opposed to two-sided) Pontryagin space.
\end{itemize}

The paper contains 6 sections, besides the Introduction. Section 2 contains a
quick survey of the Krein--Langer result in the classical case. Section 3
introduces slice hyperholomorphic functions and discusses Blaschke products.
Section 4 contains some useful results in quaternionic functional analysis,
among which Schauder-Tychonoff theorem. In Section 5 we present generalized
Schur functions and their realizations. Finally, in Section 6  we prove the
Krein-Langer factorization for generalized Schur functions defined in a subset
of the unit ball and finally, in Section 7, we state the analogous result in
the case of the half-space.
\section{A survey of the classical case}
\setcounter{equation}{0}

The celebrated one-to-one correspondence between positive
definite functions and reproducing kernel Hilbert spaces (see
\cite{aron}) extends to the indefinite case, when one considers
functions with a finite number of negative squares and
reproducing kernel Pontryagin spaces; see \cite{ad3,schwartz,
Sorjonen73}. We recall the definition of the latter for the
convenience of the reader.\smallskip

A complex vector space $\V$ endowed with a sesquilinear form
$[\cdot,\cdot]$ is called an indefinite inner product space
(which we will also denote by the pair $(\V,[\cdot,\cdot])$). The
form $[\cdot,\cdot]$ defines an orthogonality: two vectors
$v,w\in \V$ are orthogonal if $[v,w]=0$, and two linear subspaces
$\V_1$ and $\V_2$ of $\V$ are orthogonal if every vector of
$\V_1$ is orthogonal to every vector of $\V_2$. Orthogonal sums
will be denoted by the symbol $[+]$. Note that two orthogonal
spaces may intersect. We will denote by the symbol $[\oplus]$ a
direct orthogonal sum. A complex vector space $\V$ is a Krein
space if it can be written (in general in a non-unique way) as
\begin{equation}
\label{decomposition11}
\V=\V_+[\oplus]\V_-,
\end{equation}
where $(\V_+,[\cdot,\cdot])$ and
$(\V_-,-[\cdot,\cdot])$ are Hilbert spaces. When the space $\V_-$
(or, as in \cite{ikl}, the space $\V_+$) is finite dimensional
(note that this property does not depend on the decomposition),
$\V$ is called a Pontryagin space. The space $\mathcal V$ endowed with the form
\[
\langle h,g\rangle=[h_+,g_+]-[h_-,g_-],
\]
where $h=h_++h_-$ and $g=g_++g_-$ are the decompositions of
$f,g\in\mathcal V$ along \eqref{decomposition11}, is a Hilbert
space. One endows $\mathcal V$ with the corresponding topology.
This topology is independent of the decomposition
\eqref{decomposition11} (the latter is not unique, but is in the
definite case).\smallskip

Let now $T$ be a linear densely defined map from a Pontryagin space
$(\mathcal P_1,[\cdot,\cdot]_1)$ into a Pontryagin
space $(\mathcal P_2,[\cdot,\cdot]_2)$. Its adjoint is the operator
$T^*$ with domain ${\rm Dom}\, (T^*)$ defined by:
\[
\left\{g\in\mathcal P_2\,\,:\, h\mapsto\,\, [Th,g]_2\,\,\text{is continuous}\right\}.
\]
One then defines by $T^*g$ the unique element in $\mathcal P_1$ which satisfies
\[
[Th,g]_2=[h,T^*g]_1.
\]
Such an element exists by the Riesz representation theorem.\smallskip

The operator $T$ is called a contraction if
\[
[Th,Th]_2\le [h,h]_1,\quad \forall\, h\in{\rm Dom}\, (T),
\]
while it is said to be a coisometry if $TT^*=I$.

\begin{thm}
A densely defined contraction between Pontryagin spaces of the same
index has a unique contractive extension and its adjoint
is also a contraction.
\label{tmtm1}
\end{thm}

We refer to \cite{azih,bognar,MR1364446} for the theory of Pontryagin and Krein spaces,
 and of their operators.\smallskip

With these definitions, we can state the following theorem, which
gathers the main properties of generalized Schur functions.

\begin{thm}
Let $S$ be a $\mathbb C^{r\times s}$-valued function analytic in a neighborhood $\Omega$ of
the origin. Then the following are equivalent:\\
$(1)$ The kernel $K_S(z,w)$ has a finite number of negative squares in $\Omega$.\\
$(2)$ There is a Pontryagin space $\mathcal P$ and a coisometric operator matrix
\[
\begin{pmatrix}A&B\\ C&D\end{pmatrix}\ : \ \mathcal P\oplus \mathbb C^s \to \mathcal P\oplus \mathbb C^r
\]
such that
\begin{equation}
\label{realco}
S(z)=D+zC(I-zA)^{-1}B,\quad z\in\Omega.
\end{equation}
$(3)$ There exists a $\mathbb C^{r\times s}$-valued Schur
function $S_0$ and a $\mathbb C^{r\times r}$-valued Blaschke
product $B_0$ such that
\[
S(z)=B_0(z)^{-1}S_0(z),\quad z\in\Omega.
\]
\end{thm}

As a corollary we note that $S$ can be extended to a function of
bounded type in $\mathbb D$, with boundary limits almost
everywhere of norm less than or equal $1$.\\

\noindent We note the following:\\
$(a)$  When the pair $(C,A)$ is observable, meaning
\begin{equation}
\label{obsr}
\cap_{n=0}^\infty\ker CA^n=\left\{0\right\},
\end{equation}
the realization \eqref{realco} is unique, up to an isomorphism of
Pontryagin spaces.\\
$(b)$ One can take for $\mathcal P$ the reproducing kernel
Pontryagin space $\mathcal P(S)$ with reproducing kernel $K_S$.
When $0\in\Omega$ we have the backward shift realization
\[
\begin{split}
Af&=R_0f,\\
Bc&=R_0Sc,\\
Cf&=f(0),\\
Dc&=S(0) c,
\end{split}
\]
where $f\in\mathcal P(S)$, $c\in\mathbb C^s$ and where
$R_0$ denotes the backward shift operator
\[
R_0f(z)=\begin{cases}\,\, \displaystyle\frac{f(z)-f(0)}{z},\,\, z\not=0,\\
                     \,\, f^\prime(0),\,\,\,\,\,\,\,\,\,\, z=0.
                     \end{cases}
\]
See \cite{adrs} for more details on this construction, and on the related
isometric and unitary realizations.

\section{Slice hyperholomorphic functions and Blaschke products}

Let $\mathbb H$ be the real associative algebra of quaternions, where a
quaternion $p$ is denoted by $p=x_0+ix_1+jx_2+kx_3$, $x_i\in \mathbb{R}$,
and the elements $\{1, i,j,k \}$
satisfy the relations
$
i^2=j^2=k^2=-1,\
 ij =-ji =k,\
jk =-kj =i ,
 \  ki =-ik =j .
$
As is customary, $\bar p=x_0-ix_1-jx_2-kx_3$ is called the conjugate of $p$, the
real part $x_0=\frac 12 (p+\bar p)$ of a quaternion is also denoted by ${\rm Re}(p)$,
while $|p|^2=p\overline{p}$.
 The symbol $\mathbb{S}$ denotes the 2-sphere of purely imaginary unit
 quaternions, i.e.
$$
\mathbb{S}=\{ p=ix_1+jx_2+kx_3\ |\  x_1^2+x_2^2+x_3^2=1\}.
$$
If $I\in\mathbb S$ then $I^2=-1$ and  any nonreal quaternion
$p=x_0+ix_1+jx_2+kx_3$ uniquely determines an element
$I_p=(ix_1+jx_2+kx_3)/|ix_1+jx_2+kx_3|\in\mathbb S$.  \\
(We note that later $i,j,k$ may also denote some indices, but the context will make clear the use of the notation).\\
Let $\mathbb C_I$ be the
complex plane $\mathbb{R}+I\mathbb{R}$ passing through $1$
and $I$ and let $x+Iy$ be an element on $\mathbb{C}_I$. Any $p=x+Iy$ defines a 2-sphere $[p]=\{x+Jy\, \, :\, J\in\mathbb S\}$.\\
We now recall the notion of slice hyperholomorphic function:
\begin{defn}
Let $\Omega\subseteq\mathbb H$ be an open set and let
$f:\ \Omega\to\mathbb H$ be a real differentiable function. Let
$I\in\mathbb{S}$ and let $f_I$ be the restriction of $f$ to the
complex plane $\mathbb{C}_I$.
 We say that $f$ is a (left) slice hyperholomorphic function
  in $\Omega$ if, for every
$I\in\mathbb{S}$, $f_I$ satisfies
$$
\frac{1}{2}\left(\frac{\partial }{\partial x}+I\frac{\partial
}{\partial y}\right)f_I(x+Iy)=0.
$$
 We say that $f$ is a right slice hyperholomorphic function
 in $\Omega$ if, for every
$I\in\mathbb{S}$, $f_I$ satisfies
$$
\frac{1}{2}\left(\frac{\partial }{\partial x}f_I(x+Iy)+\frac{\partial
}{\partial y}f_I(x+Iy) I\right)=0.
$$
\end{defn}
The set of  slice hyperholomorphic functions on $\Omega$ will be
denoted by $\mathcal R (\Omega)$. It is a right linear space on
$\mathbb H$.\\
Slice hyperholomorphic functions possess good
properties when they are defined on the so called axially
symmetric  slice domains defined below.
\begin{defn}
Let $\Omega$ be a domain in $\mathbb{H}$.
We say that $\Omega$ is a
slice domain (s-domain for short) if $\Omega \cap \mathbb{R}$ is non empty and if
$\Omega\cap \mathbb{C}_I$ is a domain in $\mathbb{C}_I$ for all $I \in \mathbb{S}$.
We say that $\Omega$ is
axially symmetric if, for all $q \in \Omega$, the
sphere $[q]$ is contained in $\Omega$.
\end{defn}
A function $f$ slice hyperholomorphic on an axially symmetric s-domain $\Omega$ is determined by its restriction to any complex plane $\mathbb C_I$, see \cite[Theorem 4.3.2]{MR2752913}.
\begin{thm}[Structure formula]
Let $\Omega\subseteq\mathbb H$ be an axially symmetric s-domain, and let $f\in\mathcal R(\Omega)$. Then for any $x+Jy\in\Omega$ the following formula holds
\begin{equation}\label{repr}
f(x+J y)=\frac 12 \left[ f(x+Iy) +f(x-Iy) + J I
(f(x-Iy)-f(x+Iy))\right].
\end{equation}
\end{thm}
As a consequence of this result, we have the following definition:
 \begin{defn}
 Let $\Omega$ be an axially symmetric s-domain.
  Let $h:\ \Omega\cap \mathbb{C}_I\to \mathbb H$ be a holomorphic map. Then it admits a (unique) left slice hyperholomorphic extension ${\rm ext}(h): \ \Omega \to\mathbb H$ defined by:
\begin{equation}\label{ext}
{\rm ext}(h)(x+J y)=\frac 12 \left[ h(x+Iy) +h(x-Iy) + J I
(h(x-Iy)-h(x+Iy))\right].
\end{equation}
\end{defn}
\begin{rem}{\rm
Let $\Omega\subseteq\mathbb H$ be an axially symmetric s-domain and let
$f,g\in\mathcal R(\Omega)$. We can define a suitable product, called the
$\star$-product, such that the resulting function $f\star g$ is slice
hyperholomorphic. We first define a product between the restrictions $f_I$,
$g_I$ of $f$, $g$ to $\Omega\cap\mathbb C_I$. This product can be extended to the whole $\Omega$ using formula (\ref{ext}). Outside the spheres associated with the zeroes of $f\in\mathcal R(\Omega)$ we can consider its slice regular inverse $f^{-\star}$. Note also
that $(f\star g)^{-\star}=g^{-\star}\star f^{-\star}$ where it is defined.
We refer the reader to \cite[p. 125-129]{MR2752913}  for the details on the
$\star$-product and $\star$-inverse.
}
\end{rem}
The  $\star$-product can be related to the pointwise product as described in
the following result, \cite[Proposition
4.3.22]{MR2752913}:
\begin{prop}\label{pointwiseproduct}
Let $\Omega \subseteq \mathbb{H}$ be an axially symmetric s-domain,
 $f, g : \Omega \to \mathbb{H}$ be slice hyperholomorphic functions. Then
\begin{equation}
\label{productfg}
(f \star g)(p) = f(p) g(f(p)^{-1}pf(p)),
\end{equation}
for all $p\in \Omega$, $f(p)\not=0$, while $(f \star g)(p) = 0$ when $p\in \Omega$, $f(p)=0$.
\end{prop}
An immediate consequence is the following:
\begin{cor}\label{zeroes}
If $(f\star g)(p)=0$ then either $f(p)=0$ or $f(p)\not=0$ and $g(f(p)^{-1}pf(p))$ $=0$.
\end{cor}
\begin{rem}\label{zeromult}{\rm
Corollary \ref{zeroes} applies in particular to polynomials, allowing to recover a well known result, see \cite{Lam_book}: if a polynomial $\mathcal{Q}(p)$ factors as
\begin{equation}\label{factor}
\mathcal{Q}(p)=(p-\alpha_1)\star \ldots \star (p-\alpha_n),\quad \alpha_{j+1}\not=\bar{\alpha}_{j}, j=1,\ldots, n-1
 \end{equation}
   then $\alpha_1$ is a root of $\mathcal{Q}(p)$ while all the other zeroes $\tilde\alpha_j$, $j=2,\ldots, n$ belong to the spheres $[\alpha_j]$, $j=2,\ldots, n$. The decomposition of the polynomial $\mathcal Q$, in general, is not unique.\\
    Note that when $\alpha_{j+1}=\bar{\alpha}_{j}$ then $\mathcal{Q}(p)$ contains the second degree factor $p^2+2{\rm Re}(\alpha_j)p+|\alpha_j|^2$ and the zero set of  $\mathcal{Q}(p)$ contains the whole sphere $[\alpha_j]$. We will say that $[\alpha_j]$ is a spherical zero of the polynomial $\mathcal{Q}$.
}
\end{rem}
\begin{rem}{\rm
Assume that $\mathcal Q(p)$ factors as in (\ref{factor}) and assume that $\alpha_j\in [\alpha_1]$ for all $j=2,\ldots, n$. Then the only root of $\mathcal{Q}(p)$ is $p=\alpha_1$, see \cite[Lemma 2.2.11]{PhDPereira}, \cite[p. 519]{MR2205693}  the decomposition in linear factors is unique, and $\alpha_1$ is the only root of $\mathcal{Q}$. \\
Assume that $[\alpha_j]$ is a spherical zero. Then, for any $a_j\in [\alpha_j]$ we have
$$
p^2+2{\rm Re}(\alpha_j)p+|\alpha_j|^2=(p-a_j)\star (p-\bar a_j)=(p-\bar a_j)\star (p-a_j)
$$
thus showing that both $a_j$ and $\bar a_j$ are zeroes of multiplicity 1. So we can say that the (points of the) sphere $[\alpha_j]$ have multiplicity $1$. Thus the multiplicity of a spherical zero $[\alpha_j]$ equals the exponent of  $p^2+2{\rm Re}(\alpha_j)p+|\alpha_j|^2$ in a factorization of $\mathcal{Q}(p)$.\\
}
\end{rem}
The discussion in the previous remark justifies the following:
\begin{defn}\label{mult}
Let
$$
\mathcal{Q}(p)=(p-\alpha_1)\star \ldots \star (p-\alpha_n),\quad \alpha_{j+1}\not=\bar{\alpha}_{j},\ \  j=1,\ldots, n-1.
$$
We say that $\alpha_1$ is a zero of $\mathcal Q$ of {\em multiplicity} $1$ if $\alpha_j\not\in [\alpha_1]$ for $j=2,\ldots, n$.\\
We say that $\alpha_1$ is a zero of $\mathcal Q$ of {\em multiplicity} $n\geq 2$ if $\alpha_j\in [\alpha_1]$ for all $j=2,\ldots, n$.\\
Assume now that $\mathcal{Q}(p)$ contains the factor $(p^2+2{\rm Re}(\alpha_j)p+|\alpha_j|^2)$ and $[\alpha_j]$ is a zero of $\mathcal Q(p)$. We say that the {\em multiplicity of the spherical zero} $[\alpha_j]$ is $m_j$ if  $m_j$ is the maximum of the integers $m$ such that $(p^2+2{\rm Re}(\alpha_j)p+|\alpha_j|^2)^{m}$ divides $\mathcal{Q}(p)$.
\end{defn}
Note that the notion of multiplicity of a spherical zero given in \cite{gszero} is different since, under the same conditions described in Definition \ref{mult}, it would be $2m_j$.
\begin{rem}
The polynomial $\mathcal{Q}(p)$ can be factored as follows, see e.g. \cite[Theorem 2.1]{gszero}:
\[
\mathcal{Q}(p)=\prod_{j=1}^r (p^2+2{\rm Re}(\alpha_j)p+|\alpha_j|^2)^{m_j}  \left( \prod_{i=1}^{\star s_{\ }} \prod_{j=1}^{\star n_i} \ (p- \alpha_{ij})\right) a,
\]
where $\displaystyle\prod^\star$ denotes the $\star$-product of the factors, $[\alpha_i]\not=[\alpha_j]$ for $i\not=j$, $\alpha_{ij}\in [a_i]$ for all $j=1,\ldots, n_i$ and $[a_{i}]\not=[a_{\ell}]$ for $i\not=\ell$.
Note that $$\deg(\mathcal{Q})=\sum_{j=1}^r 2m_j+\sum_{i=1}^s n_i.$$
\end{rem}
\begin{defn}
Let $a\in\mathbb{H}$, $|a|<1$. The function
\begin{equation}
\label{eqBlaschke}
B_a(p)=(1-p\bar a)^{-\star}\star(a-p)\frac{\bar a}{|a|}
\end{equation}
is called a Blaschke factor at $a$.
 \end{defn}
 \begin{rem}
 Using Proposition
 \ref{pointwiseproduct}, $B_a(p)$ can be rewritten as
$$
 B_a(p)= (1-\tilde{p}\bar a)^{-1}(a-\tilde p)\frac{\bar a}{|a|}
 $$
 where  $\tilde p=(1-p a)^{-1} p (1-p a)$.
 \end{rem}
  The following result is immediate, see \cite{MR3127378}:
 \begin{prop}
 Let $a\in\mathbb{H}$, $|a|<1$. The Blaschke factor $B_a$ is a slice hyperholomorphic function in $\mathbb B$.
 \end{prop}
 As one expects, $B_a(p)$ has only one zero at $p=a$ and
 analogously to what happens in the case of the zeroes of a function, the product of two Blaschke factors of the form $B_a(p)\star B_{\bar a}(p)$ gives the Blaschke factor with zeroes at the sphere $[a]$. Thus we give the following definition:
\begin{defn}
Let $a\in\mathbb{H}$, $|a|<1$. The function
\begin{equation}
\label{blas_sph}
B_{[a]}(p)=(1-2{\rm Re}(a)p+p^2|a|^2)^{-1}(|a|^2-2{\rm Re}(a)p+p^2)
\end{equation}
is called  Blaschke factor at the sphere $[a]$.
 \end{defn}
 Theorem 5.16 in \cite{MR3127378} assigns a Blaschke product having zeroes at a given set of points $a_j$ with multiplicities $n_j$, $j\geq 1$ and at spheres $[c_i]$ with multiplicities $m_i$, $i\geq 1$, where the multiplicities are meant as exponents of the factors $(p-a_j)$ and $(p^2-{\rm Re}(a_j) p+|a_j|^2)$, respectively. In view of Definition \ref{mult}, the polynomial $(p-a_j)^{\star n_j}$ is not the unique polynomial having a zero at $a_j$ with the given multiplicity $n_j$, thus the Blaschke product $\prod^{\star n_j}_{j=1} B_{a_j}$ is not the unique Blaschke product having zero at $a_j$ with multiplicity $n_j$.\\
  We give below a form of Theorem  5.16 in \cite{MR3127378} in which we use the notion of multiplicity in Definition \ref{mult}:
 \begin{thm}\label{blaschke zeros}
A Blaschke product having zeroes at the set
 $$
 Z=\{(a_1,n_1),  \ldots, ([c_1],m_1), \ldots \}
 $$
 where $a_j\in \mathbb{B}$, $a_j$ have respective multiplicities $n_j\geq 1$, $a_j\not=0$ for $j=1,2,\ldots $, $[a_i]\not=[a_j]$ if $i\not=j$, $c_i\in \mathbb{B}$, the spheres $[c_j]$ have respective multiplicities $m_j\geq 1$,
 $j=1,2,\ldots$, $[c_i]\not=[c_j]$ if $i\not=j$
and
\begin{equation}\label{conditionconvergence}
\sum_{i,j\geq 1} \Big(n_i (1-|a_i|)+
2 m_j(1-|c_j|)\Big)<\infty
\end{equation}
is of the form
\[
\prod_{i\geq 1} (B_{[c_i]}(p))^{m_i}\prod_{i\geq 1}^\star \prod_{\ j=1}^{\star  n_i} (B_{\alpha_{ij}}(p)),
\]
where $n_j\geq 1$, $\alpha_{11} = a_1$ and $\alpha_{ij}$ are suitable elements in $[a_i]$ for $j=2,3,\ldots$.
\end{thm}
\begin{proof}
The fact that (\ref{conditionconvergence}) ensure the convergence of the product follows from \cite[Theorem 5.6]{MR3127378}. The zeroes of the pointwise product $\prod_{i\geq 1} (B_{[c_i]}(p))^{m_i}$ correspond to the given spheres with their multiplicities. Let us consider the product:
\[
\prod_{i=1}^{\star  n_1} (B_{\alpha_{i1}}(p))=B_{\alpha_{11}}(p)\star B_{\alpha_{12}}(p) \star \ldots \star B_{\alpha_{1 n_1}}(p).
\]
As we already observed in the proof of Proposition 5.10 in \cite{MR3127378} this product admits a zero at the point $\alpha_{11}=a_1$ and it is a zero of multiplicity $1$ if $n_1=1$; if $n_1\geq 2$, the other zeroes are $\tilde\alpha_{12}, \ldots, \tilde\alpha_{1 n_1}$ where $\tilde\alpha_{1j}$ belong to the sphere $[\alpha_{1j}]=[a_1]$.
This fact can be seen directly using formula  (\ref{productfg}). Thus, according to Remark \ref{zeromult}, $a_1$ is a zero of multiplicity $n_1$.
Let us now consider $r\geq 2$ and
\begin{equation}\label{rfactor}
\prod_{j=1}^{\star  n_r} (B_{\alpha_{rj}}(p))=B_{\alpha_{r1}}(p) \star \ldots \star B_{\alpha_{r n_r}}(p),
\end{equation}
and set
\[
B_{r-1}(p):= \prod_{i\geq 1}^{\star (r-1)} \prod_{j=1}^{\star  n_i} (B_{\alpha_{ij}}(p)).
\]
Then
\[
B_{r-1}(p)\star B_{\alpha_{r1}}(p)= B_{r-1}(p) B_{\alpha_{r1}}(B_{r-1}(p)^{-1}pB_{r-1}(p))
\]
has a zero  at
 $a_r$ if and only if $B_{\alpha_{r1}}(B_{r-1}(a_r)^{-1}a_rB_{r-1}(a_r))=0$, i.e. if and only if $\alpha_{r1}=B_{r-1}(a_r)^{-1}a_rB_{r-1}(a_r)$. If $n_r=1$ then $a_r$ is a zero of multiplicity $1$ while if $n_r\geq 2$, all the other zeroes of the product (\ref{rfactor}) belongs to the sphere $[a_r]$ thus, by Remark  \ref{zeromult}, the zero $a_r$ has multiplicity $n_r$.
 This completes the proof.
\end{proof}
\begin{rem}\label{remB}{\rm In the case in which one has to construct a Blaschke product having a zero at $a_i$ with multiplicity $n_i$ by prescribing the factors $(p- a_{i1})\star\cdots \star (p-a_{i n_i})$, $a_{ij}\in [a_i]$ for all $j=1,\ldots, n_i$, the factors in the Blaschke product must be chosen accordingly (see the proof of Theorem \ref{blaschke zeros}).  }
\end{rem}
\begin{prop} The $\star$-inverse of $B_a$ and $B_{[a]}$ are $B_{\bar a^{-1}}$, $B_{[a^{-1}]}$ respectively.
\end{prop}
\begin{proof} It follows from straightforward computations, by verifying that the products $B_a\star B_{\bar a^{-1}}$ and
$B_{[a]}\star B_{[a^{-1}]}$ equal $1$.
\end{proof}

\begin{defn}
A Blaschke product of the form
\begin{equation}\label{Bproduct}
B(p)=
\prod_{i=1}^r (B_{[c_i]}(p))^{m_i}\prod_{i=1}^{\star s} \prod_{\ j=1}^{\star  n_i} (B_{\alpha_{ij}}(p)),
\end{equation}
is said to have degree $d=\sum_{i=1}^r 2m_i+\sum_{j=1}^s n_j$.
\end{defn}
\begin{prop}
Let $B(p)$ be a Blaschke product as in (\ref{Bproduct}). Then $\dim(\mathcal H(B))=\deg B$.
\end{prop}
\begin{proof}
Let us rewrite $B(p)$ as
$$
B(p)=
\prod_{i=1}^r (B_{c_i}(p)\star B_{\bar c_i}(p))^{m_i}\prod_{i=1}^{\star s} \prod_{\ j=1}^{\star  n_i} (B_{\alpha_{ij}}(p))=\prod_{j=1}^{\star d} B_{\beta_j}(p),
$$
$d=\deg B$. Let us first observe that in the case in which the
factors $B_{\beta_j}$ are such that no three of the quaternions
$\beta_j$ belong to the same sphere, then the statement follows
from the fact that $\mathcal H(B_{\beta_j})$ is the span of
$(1-p\bar \beta_j)^{-\star}$. MOreover $(1-p\bar
\beta_1)^{-\star},\ldots, (1-p\bar \beta_d)^{-\star}$ are linearly
independent in the Hardy space $\mathbf H_2(\mathbb B)$, see \cite[Remark
3.1]{abcs1}. So we now assume that $d\geq 3$ and at least three
among the $\beta_j$'s belong the same sphere. We proceed by
induction. Assume that $d=3$ and $\beta_1,\beta_2,\beta_3$ belong
to the same sphere. Since
\[
\begin{split}
K_B(p,q)&=\sum_n p^n(1-B(p)B(q)^*)\bar q^n=\sum_n p^n(1-B_{\beta_1}(p)B_{\beta_1}(q)^*)\bar q^n\\
&+B_{\beta_1}(p)\star\sum_n p^n(1-B_{\beta_2}(p)B_{\beta_2}(q)^*)\bar q^n\star_r B_{\beta_1}(q)^*\\
&+B_{\beta_1}(p)\star B_{\beta_2}(p)\star\sum_n p^n(1-B_{\beta_3}(p)B_{\beta_3}(q)^*)\bar q^n\star_r B_{\beta_1}(q)^*\star_r B_{\beta_1}(q)^*\\
\end{split}
\]
we have
\begin{equation}\label{sum}
\mathcal H(B_{\beta_1})+B_{\beta_1}\star \mathcal H(B_{\beta_2})+B_{\beta_1}\star B_{\beta_2}\star \mathcal H(B_{\beta_3}).
\end{equation}
Now note that $\mathcal H(B_{\beta_1})$ is spanned by
$f_1(p)=(1-p\bar \beta_1)^{-\star}$, $B_{\beta_1}\star \mathcal
H(B_{\beta_2})$ is spanned by $f_2(p)=B_{\beta_1}(p)\star
(1-p\bar \beta_2)^{-\star}$ and, finally, $B_{\beta_1}\star
B_{\beta_2}\star \mathcal H(B_{\beta_3})$ is spanned by
$f_3(p)=B_{\beta_1}(p)\star B_{\beta_2}(p)\star (1-p\bar
\beta_3)^{-\star}$. By using the reproducing
property of $f_1$ we have $[f_1,f_2]=0$ and $[f_1,f_3]=0$ (here
$[\cdot,\cdot]$ denotes the inner product in $\mathbf H_2(\mathbb
B)$). Observe that $$[f_2,f_3]=[ (1-p\bar \beta_2)^{-\star},
B_{\beta_2}(p)\star (1-p\bar \beta_3)^{-\star}]=0$$ since the
left multiplication by $B_{\beta_1}(p)$ is an isometry in
$\mathbf H_2(\mathbb B)$ and by the reproducing property of
$(1-p\bar \beta_2)^{-\star}$. So $f_1,f_2,f_3$ are orthogonal in
$\mathbf H_2(\mathbb B)$ and so they are linearly independent. We
conclude that the sum (\ref{sum}) is direct and has dimension
$3$. Now assume that the assertion hold when $d=n$ and there in
$B(p)$ are at least three Blaschke factors at points on the same
sphere. We show that the assertion holds for $d=n+1$. We
generalize the above discussion by considering
\begin{equation}\label{nsum}
\begin{split}
&(\mathcal H(B_{\beta_1})+B_{\beta_1}\star \mathcal H(B_{\beta_2})+\cdots+B_{\beta_1}\star \cdots \star B_{\beta_{n-1}}\star \mathcal H(B_{\beta_n})+\cdots +\\
&+ B_{\beta_1}\star \cdots \star B_{\beta_{n}}\star \mathcal H(B_{\beta_{n+1}}).\\
\end{split}
\end{equation}
Let us denote, as before, by $f_1(p)=(1-p\bar \beta_1)^{-\star}$ a generator of $\mathcal H(B_{\beta_1})$ and by
$f_j(p)=B_{\beta_1}\star \cdots \star B_{\beta_{j-1}}\star (1-p\bar \beta_{j})^{-\star}$ a generator of
$B_{\beta_1}\star \cdots \star B_{\beta_{j-1}}\star \mathcal H(B_{\beta_{j}})$, $j=1,\ldots, n+1$.
By the induction hypothesis, the sum of the first $n$ terms is direct  so we show that $[f_j,f_{n+1}]=0$ for $j=1,\ldots, n$. This follows, as before, from the fact that the multiplication by a Blaschke factor is an isometry and by the reproducing property. The statement follows.
\end{proof}
We now introduce the Blaschke factors in the half--space
$$\mathbb H_+=\{p\in\mathbb H\ :\ {\rm Re}(p)>0\}.$$
\begin{defn}
For $a\in\mathbb H_+$ set
\[
b_{a}(p)=(p+\bar a)^{-\star}\star (p-a).
\]
The function $b_a(p)$ is called Blaschke factor at $a$ in the half space $\mathbb H_+$.
\end{defn}
\begin{rem}{\rm
The function $b_a(p)$ is defined outside the sphere $[-a]$
and it has a zero at $p=a$.
A Blaschke factor $b_a$ is slice hyperholomorphic in $\mathbb H^+$.}
\end{rem}
As before, we can also introduce Blaschke factors at spheres:
\begin{defn}
For $a\in\mathbb H_+$ set
\[
b_{[a]}(p)=(p^2 +2 {\rm Re}(a)p +|a|^2)^{-1} (p^2 -2 {\rm Re}(a)p +|a|^2).
\]
The function $b_a(p)$ is called Blaschke factor at the sphere $[a]$ in the half space $\mathbb H_+$.
\end{defn}
We now state the following result
whose proof mimics the lines of the proof of Theorem \ref{blaschke zeros} with obvious changes.
Note that an analog of Remark \ref{remB} holds also in this case.

\begin{thm}
A Blaschke product having zeroes at the set
 $$
 Z=\{(a_1,n_1),  \ldots,  ([c_1],m_1),  \ldots \}
 $$
 where $a_j\in \mathbb{H}_+$, $a_j$ have respective multiplicities $n_j\geq 1$, $[a_i]\not=[a_j]$ if $i\not=j$, $c_i\in \mathbb{H}_+$, the spheres $[c_j]$ have respective multiplicities $m_j\geq 1$,
 $j=1,2,\ldots$, $[c_i]\not=[c_j]$ if $i\not=j$
and
$$
\sum_{i,j\geq 1} \Big(n_i (1-|a_i|)+
2m_j(1-|c_j|)\Big)<\infty
$$
is given by
\[
\prod_{i\geq 1} (b_{[c_i]}(p))^{m_i}\prod_{i\geq 1}^\star \prod_{j=1}^{\star n_i} (b_{\alpha_{ij}}(p)),
\]
where $\alpha_{11}=a_1$ and $\alpha_{ij}$ are suitable elements in $[a_i]$ for $i=2,3,\ldots$.
\end{thm}

Let $f(p)=
\sum_{n=-\infty}^{+\infty} (p-p_0)^{\star n} a_n$ where $a_n\in
\mathbb H$.
Following the standard nomenclature and
\cite{MR2955794} we now give the definition of singularity of a slice regular function:
\begin{defn}
A function $f$
  has a pole at the point $p_0$ if there exists $m\geq 0$ such that $a_{-k}=0$ for $k>m$.   The minimum of such $m$ is called the order of the pole;\\
  If $p$ is not a pole then we call it an essential singularity for $f$;\\
  $f$ has a removable singularity at $p_0$ if it can be extended in a neighborhood of $p_0$ as a slice hyperholomorphic function.\\
\end{defn}
A function $f$ has a pole at $p_0$ if and only if its restriction to a complex plane has a pole. In this framework there can be poles of order $0$. To give an example, let $I\in\mathbb S$; then  the function $(p+I)^{-\star}=(p^2+1)^{-1}(p-I)$ has a pole of order $0$ at the point $-I$ which, however, is not a removable singularity, see \cite[p.55]{MR2752913}.

\begin{defn} Let $\Omega$ be an axially symmetric s-domain in $\mathbb H$. We say that a function $f:\, \Omega\to \mathbb H$ is
slice hypermeromorphic in $\Omega$ if $f$ is slice hyperholomorphic in
$\Omega'\subset \Omega$ such that every point in $\Omega\setminus\Omega'$ is a pole.
\end{defn}

\section{Some results from quaternionic functional analysis}
\setcounter{equation}{0}
The tools from quaternionic functional analysis needed in the
present paper are of two kinds. On one hand, we need some results from the
theory of quaternionic Pontryagin spaces, taken
essentially from \cite{as3}. On the other hand, we also need the quaternionic
version of the Schauder-Tychonoff theorem in order to prove an
invariant subspace theorem for contractions in Pontryagin spaces.
More generally we note that in our on-going project on Schur
analysis in the slice hyperholomorphic setting we were lead to
prove a number of results in quaternionic functional analysis not
readily available in the literature.\smallskip

%\subsection{Quaternionic Pontryagin spaces}
Operator theory in (quaternionic) Pontryagin spaces plays an important role
in (quaternionic) Schur analysis, and we here recall some definitions
and results needed in the sequel. We refer to \cite{as3} for more information.

\begin{defn}\label{innerproduct}
Let ${\mathcal V}$ be a right quaternionic vector space. The map
\[
[\cdot, \cdot]\,\, :\,\, {\mathcal V}\times {\mathcal
V}\quad\longrightarrow\quad \mathbb H
\]
is called an inner product if it is a (right) sesquilinear form:
\[
[v_1c_1,v_2c_2]=\overline{c_2}[v_1,v_2]c_1,\quad \forall
v_1,v_2\in {\mathcal V},\,\,{and}\,\, c_1,c_2\in\mathbb H,
\]
which is Hermitian in the sense that:
\[
[v,w]=\overline{[w,v]},\quad \forall v,w\in\mathbb {\mathcal V}.
\]
\label{dn:hermi}
\end{defn}

A quaternionic inner product space ${\mathcal V}$ is called a
Pontryagin space if it can be written as a direct and orthogonal
sum
\begin{equation}
\label{decomposition1}
{\mathcal V}={\mathcal V}_+[\oplus]{\mathcal V}_-,
\end{equation}
where $({\mathcal V}_+, [\cdot,\cdot])$ is a Hilbert space, and
$({\mathcal V}_-, -[\cdot,\cdot])$ is a finite dimensional
Hilbert space.
As in the complex case, the space $\mathcal V$ endowed with the form
\begin{equation}
\label{eqnorm}
\langle h,g\rangle=[h_+,g_+]-[h_-,g_-],
\end{equation}
where $h=h_++h_-$ and $g=g_++g_-$ are the decompositions of $f,g\in\mathcal V$ along
\eqref{decomposition1}, is a Hilbert space and the norms associated with the inner
products \eqref{eqnorm}
are equivalent, and hence define the same topology. The notions of adjoint and
contraction are defined as in
the complex case, and Theorem \ref{tmtm1} still holds in the quaternionic setting:

\begin{thm}\cite[Theorem 7.2]{acs3}
\label{tmtm2}
A densely defined contraction between quaternionic
Pontryagin spaces of the same index has a unique contractive extension and its adjoint
is also a contraction.
\end{thm}

A key result used in the proof of the Krein-Langer factorization is the following
invariant subspace theorem.

\begin{thm}\cite[Theorem 4.6]{2013arXiv1310.1035A}
A contraction in a quaternionic Pontryagin space has a unique
maximal invariant negative subspace, and it is one-to-one on it.
\label{tm:inv}
\end{thm}

The arguments there follow the ones given in the complex case in
\cite{MR92m:47068}, and require in particular to prove first a
quaternionic version of the Schauder-Tychonoff theorem, and an
associated lemma. We recall these for completeness:

\begin{lem}\cite[Lemma 4.4]{2013arXiv1310.1035A}\label{previous}
Let $\mathcal{K}$ be a compact convex subset of a locally convex
linear quaternionic space $\mathcal{V}$ and let $T:\mathcal{K}\to
\mathcal{K}$ be continuous. If $\mathcal{K}$ contains at least
two points, then there exists a proper closed convex subset
$\mathcal{K}_1 \subset \mathcal{K}$ such that
$T(\mathcal{K}_1)\subseteq \mathcal{K}_1$.
\end{lem}

\begin{thm}[Schauder-Tychonoff] \cite[Theorem 4.5]{2013arXiv1310.1035A}
A  compact convex subset of a  locally convex quaternionic
linear space has the fixed point property.
\end{thm}

\section{Generalized Schur functions and their realizations}
\setcounter{equation}{0}

The definition of negative squares makes sense in the
quaternionic setting since an Hermitian quaternionic matrix $H$
is diagonalizable: it can be written as $H=UDU^*$, where $U$ is
unitary and $D$ is unique and with real entries. The number of
strictly negative eigenvalues of $H$ is exactly the number of
strictly negative elements of $D$, see \cite{MR97h:15020}. The
one-to-one correspondence between reproducing kernel Pontryagin
spaces and functions with a finite number of negative squares,
proved in the classical case by \cite{schwartz,Sorjonen73},
extends to the Pontryagin space setting, see \cite{as3}.\smallskip

We first recall a definition. A quaternionic matrix $J$ is called a
signature matrix if it is both self-adjoint and unitary. The
index of $J$ is the number of strictly negative eigenvalues of
$J$, and the latter is well defined because of the spectral theorem for
quaternionic matrices. See e.g. \cite{MR97h:15020}.

\begin{defn}
\label{def1} Let $\Omega$ be an axially symmetric s-domain
contained in the unit ball, let $J_1\in\mathbb H^{s\times s}$ and
$J_2\in\mathbb H^{r\times r}$ be two signature matrix of the same
index, and let $S$ be a $\mathbb H^{r\times s}$-valued function,
slice hyperholomorphic in $\Omega$. Then $S$ is called a
generalized Schur function if the kernel
\[
K_S(p,q)=\sum_{n=0}^\infty p^n(J_2-S(p)J_1S(q)^*)\overline{q}^n
\]
has a finite number of negative squares, say $\kappa$, in $\Omega$.
We set $\kappa={\rm ind}\, S$ and call it the index of $S$.
\end{defn}

We will denote by $\mathscr S_\kappa(J_1,J_2)$ the family of
generalized Schur functions of index $\kappa$.

\begin{lem}\label{La 5.2}
In the notation of Definition \ref{def1},
let $x_0\in\Omega\cap\mathbb R$. Let $b(p)=\frac{p+x_0}{1+px_0}$. Then the function
$S\circ b$ is a generalized Schur function slice hyperholomorphic
at the origin and with the same index as $S$.
\end{lem}

\begin{proof}
First of all, we note that $(1+px_0)^{-\star}=(1+px_0)^{-1}$ since $x_0\in\mathbb R$, and that
$(1+px_0)^{-1}$ commute with $p+x_0$ thus the rational function $b(p)$ is well defined.
The result then follows from the formula
\begin{equation}\label{sameindex}
\begin{split}
\sum_{n=0}^\infty p^n(J_2-S(b(p))J_1S(b(q))^*)\overline{q}^n&=
(1-x_0^2)\times\\
&\hspace{-5cm}\times(1+px_0)^{-1}\left(\sum_{n=0}^\infty
b(p)^n(J_2-S(b(p))J_1S(b(q))^*)\overline{b(q)}^n\right)(1+\overline{q}x_0)^{-1}.
\end{split}
\end{equation}
To show the validity of (\ref{sameindex}) we use \cite[Proposition 2.22]{2013arXiv1310.1035A} to compute the left hand side which gives
\begin{equation}\label{sameindexleft}
\sum_{n=0}^\infty p^n(J_2-S(b(p))J_1S(b(q))^*)\overline{q}^n=
(J_2-S(b(p))J_1S(b(q))^*)\star (1-p\bar q)^{-\star},
\end{equation}
where the $\star$-product is the left one and it is computed with respect to $p$.
The right hand side of (\ref{sameindex}) can be computed in a similar was and gives
\begin{equation}\label{sameindexright}
\begin{split}
&(1-x_0^2)(1+px_0)^{-1}\left(\sum_{n=0}^\infty
b(p)^n(J_2-S(b(p))J_1S(b(q))^*)\overline{b(q)}^n\right)(1+\overline{q}x_0)^{-1}\\
& (1-x_0^2)(1+px_0)^{-1} (J_2-S(b(p))J_1S(b(q))^*)\star (1-b(p)\overline{b(q)})^{-\star}    (1+\overline{q}x_0)^{-1}.\\
\end{split}
\end{equation}
We now note that
\[
(1-b(p)\overline{b(q)})^{-\star}=\left(1- \frac{p+x_0}{1+px_0}\frac{\bar q+x_0}{1+\bar q x_0}\right)^{-\star}\\
= \frac{1}{1-x_0^2}(1+px_0)(1-p\bar q)^{-\star}(1+\bar q x_0)
\]
and substituting this expression in (\ref{sameindexright}), and using the property that
$$
(J_2-S(b(p))J_1S(b(q))^*)\star (1+px_0)=(1+px_0)(J_2-S(b(p))J_1S(b(q))^*)
$$
since $x_0\in\mathbb R$,
we obtain
\[
\begin{split}
(1-x_0^2)&(1+px_0)^{-1} (J_2-S(b(p))J_1S(b(q))^*)\\
&\star  \frac{1}{1-x_0^2}(1+px_0)(1-p\bar q)^{-\star}(1+\bar q x_0)   (1+\overline{q}x_0)^{-1}\\
&(1+px_0)^{-1} (1+px_0) (J_2-S(b(p))J_1S(b(q))^*)(1-p\bar q)^{-\star}\\
&= (J_2-S(b(p))J_1S(b(q))^*)(1-p\bar q)^{-\star}
\end{split}
\]
and the statement follows.
\end{proof}

The reproducing kernel Pontryagin space $\mathcal P(S)$ associated
with a generalized Schur function $S$, namely the space with reproducing kernel $K_S$, is a right quaternionic
vector space, with functions taking values in a two-sided
quaternionic vector space. To present the counterpart of
\eqref{realco} with $\mathcal P(S)$ as a state space we first
recall the following result, see \cite[Proposition
2.22]{2013arXiv1310.1035A}.

\begin{prop}
Let $A$ be a bounded linear operator from a right-sided
quaternionic Hilbert $\mathcal P$ space into itself, and let $C$
be a bounded linear operator from $\mathcal P$ into $\mathcal C$,
where $\mathcal C$ is a two sided quaternionic Hilbert space. The
slice hyperholomorphic extension of $C(I-xA)^{-1}$,
$1/x\in\rho_S(A)\cap\mathbb{R}$, is
\[
(C-\overline{p}CA)(I-2{\rm Re}(p)\, A +|p|^2A^2)^{-1}.
\]
\label{formula060813}
\end{prop}

\noindent We will use the notation
\begin{equation}
C\star (I-pA)^{-\star}\stackrel{\rm def.}{=}(C-\overline{p}CA)(I-2{\rm
Re}(p)\, A +|p|^2A^2)^{-1}.
\end{equation}

For the following result see \cite{acs1,MR3127378}. First two
remarks: in the statement, an observable pair is defined, as in
the complex case, by \eqref{obsr}. Next, we denote by $M^*$ the
adjoint of a quaternionic bounded linear operator from a Pontryagin space $\mathcal
P_1$ into a Pontryagin space $\mathcal P_2$:
\[
[Mp_1\,,\,p_2 ]_{\mathcal P_2}= [ p_1\,,\,M^*p_2
]_{\mathcal P_1},\quad p_1\in\mathcal P_1\quad{\rm
and}\quad p_2\in\mathcal P_2.
\]

\begin{thm}
Let $J_1\in\mathbb H^{s\times s}$ and $J_2\in\mathbb H^{r\times
r}$ be two signature matrix of the same index, and let $S$ be slice
hyperholomorphic in a neighborhood of the origin. Then, $S$ is in
$\mathscr S_\kappa(J_1,J_2)$ if and only if it can written in the
form
\begin{equation}\label{realSB}
S(p)=D+p C\star(I_{\mathcal P}-pA)^{-\star}B,
\end{equation}
where $\mathcal P$ is a right quaternionic Pontryagin space of
index $\kappa$, the pair $(C,A)$ is observable, and the
operator matrix
\begin{equation}
\label{eqM}
M=\begin{pmatrix}A&B\\ C&D\end{pmatrix}: \,\,\, \mathcal
P\oplus \mathbb H^s\,\,\,\longrightarrow\,\,\, \mathcal
P\oplus \mathbb H^r
\end{equation}
satisfies
\begin{equation}
\label{milano2013}
\begin{pmatrix}A&B\\ C&D\end{pmatrix}\begin{pmatrix} I_{\mathcal P}&0\\
0&J_1\end{pmatrix}\begin{pmatrix}A&B\\ C&D\end{pmatrix}^*=
\begin{pmatrix} I_{\mathcal P}&0\\
0&J_2\end{pmatrix}.
\end{equation}
\label{tm:real}
\end{thm}

The space $\mathcal P$ can be chosen to be the reproducing kernel
Pontryagin space $\mathcal P(S)$ with reproducing kernel $K_S(p,q)$. Then the
coisometric colligation \eqref{eqM} is given by:

\begin{equation}
\label{eq:dbh}
\begin{split}
(Af)(p)&=\begin{cases}p^{-1}(f(p)-f(0)),\quad p\not =0,\\
f_1,\quad\hspace{2.6cm}
 p=0,
\end{cases}\\
 (Bv)(p)&=\begin{cases}p^{-1}(S(p)-S(0))v,
 \quad p\not =0,\\
s_1v,\quad\hspace{2.6cm}p=0,
 \end{cases}\\
 Cf&=f(0),\\
 Dv&=S(0)v,
\end{split}
\end{equation}
where $v\in\mathbb H^s$, $S(p)=\sum_{n=0}^\infty p^ns_n$ and $f\in\mathcal P$ with power series
$f(p)=\sum_{n=0}^\infty p^nf_n$ at the origin.\\

Assume now in the previous theorem that $r=s$, $J_1=J_2=J$, and
that ${\rm dim}\, \mathcal P(S)$ is finite.  Then,
equation \eqref{milano2013} is an equality in finite dimensional
spaces (or as matrices) and the function $S$ is called
$J$-unitary. The function $S$ is moreover rational and its McMillan
degree, denoted by ${\rm deg}\, S$, is the dimension of the space
$\mathcal P(S)$ (we refer to \cite{acs1} for the notion of
rational slice-hyperholomorphic functions. Suffices here to say
that the restriction of $S$ to the real axis is a $\mathbb
H^{r\times r}$-valued rational function of a real
variable).\smallskip

The $\star$-factorization $S=S_1\star S_2$ of $S$ as a $\star$-product of two $\mathbb H^{r\times r}$-valued $J$-unitary
functions is called minimal if ${\rm deg}\, S={\rm deg}\,
S_1+{\rm deg}\, S_2$. When $\kappa=0$, $S$ is a minimal product
of elements of three types, called Blaschke-Potapov factors, and
was first introduced by V. Potapov in \cite{potapov} in the complex
case. We give now a formal definition of the Blaschke-Potapov factors:

\begin{defn}
A $\mathbb H^{r\times r}$-valued Blaschke-Potapov factor of the
first kind (resp, second kind) is defined as:
$$
B_{a}(p,P)=I_r+(B_a(p)-I_r)P
$$
where $|a|<1$ (resp. $|a|>1$) and $J,P\in\mathbb H^{r\times r}$, $J$
being a signature matrix, and $P$ a matrix such that $P^2=P$ and
$JP\geq 0$.\\
A $\mathbb H^{r\times r}$-valued Blaschke-Potapov factor of the
third kind is defined as:
$$
I_r-ku\star(p+w_0)\star(p-w_0)^{-\star} u^*J
$$
where $u\in\mathbb H^r$ is $J$-neutral (meaning $uJu^*=0$),
$|w_0|=1$ and $k>0$.
\end{defn}
\begin{rem}{\rm  In the setting of circuit theory, Blaschke-Potapov
factors of the third kind are also called Brune sections, see
e.g. \cite{MR85m:94012}, \cite{abdd}.}
\end{rem}

In the sequel, by Blaschke product we mean the product of
Blaschke-Potapov factors.
Note that the inverse of a Blaschke-Potapov factor of the
first kind is $B_{a}^{-\star}(p,P)=I_r+(B_a(p)^{-\star}-I_r)P$. \\
\smallskip

When $\kappa>0$ there need not exist minimal factorizations. We
refer to \cite{ad3,ag} for examples in the complex-valued case.
On the other hand, still when $\kappa>0$ but for $J=I_r$, a
special factorization exists, as a $\star$-quotient of two
Blaschke products. This is a special case of the factorization of
Krein-Langer. The following result plays a key role in the proof
of this factorization. It is specific of the case $J_1=I_s$ and
$J_2=I_r$, which allows us to use the fact that the adjoint of a
contraction between quaternionic Pontryagin spaces of the same index
is still a contraction.

\begin{prop}
\label{pn:acont} In the notation of Theorem \ref{tm:real}, assume
$J_1=I_s$ and $J_2=I_r$. Then the operator $A$ is a Pontryagin
contraction.
\end{prop}
\begin{proof}
Equation \eqref{milano2013} expresses that the operator matrix
$M$ (defined by \eqref{eqM}) is a coisometry, and in particular a
contraction, between Pontryagin spaces of same index. Its adjoint
is a Pontryagin space contraction (see \cite{as3}) and we have
\[
\begin{pmatrix}A&B\\ C&D\end{pmatrix}^*\begin{pmatrix} I_{\mathcal P}&0\\
0&I_r\end{pmatrix}\begin{pmatrix}A&B\\ C&D\end{pmatrix}\le
\begin{pmatrix} I_{\mathcal P}&0\\
0&I_s\end{pmatrix}.
\]
It follows from this inequality that
\begin{equation}
\label{eq:acont}
A^*A+C^*C\le I_s.
\end{equation}
Since the range of $C$ is inside the Hilbert space $\mathbb H^r$
we have that $A^*$ is a contraction from $\mathcal P$ into itself,
and so is its adjoint $A=(A^*)^*$.
\end{proof}
\section{The factorization theorem}
\setcounter{equation}{0} Below we prove a version of the
Krein-Langer factorization theorem in the slice hyperholomorphic
setting which generalizes \cite[Theorem 9.2]{acs3}. The role of
the Blaschke factors $B_a$ in the scalar case is played here by
the Blaschke-Potapov factors with $J=I$.

\begin{thm}
Let $J_1=I_s$ and $J_2=I_r$, and let $S$ be a $\mathbb H^{r\times
s}$-valued generalized Schur function of index $\kappa$. Then
there exists a $\mathbb H^{r\times r}$-valued Blaschke product
$B_0$  of degree $\kappa$ and a
 $\mathbb H^{r\times s}$-valued  Schur function $S_0$
such that
\[
S(p)=(B_0^{-\star}\star S_0)(p).
\]
\end{thm}
\begin{proof} We proceed in a number of steps:\\

STEP 1: {\sl One can assume that $S$ is slice hyperholomorphic at
the origin.}\smallskip

To check this, we note that  whenever $f=g\star h$, we have $f\circ b=(g\circ b)\star (h\circ b)$ where $b(p)=\frac{p+x_0}{1+px_0}$, $x_0\in\mathbb R$.
This equality is true on $\Omega\cap \mathbb R_+$, and extends to
$\Omega$ by slice hyperholomorphic extension. Thus, taking into account Lemma \ref{La 5.2}, we now assume $0\in\Omega$.\\

STEP 2: {\sl Let \eqref{realSB} be a coisometric realization of
$S$. Then $A$ has a unique maximal strictly negative invariant
subspace $\mathcal M$.}\smallskip

Indeed, $A$ is a contraction as proved in Proposition
\ref{pn:acont}. The result then follows from
Theorem \ref{tm:inv}.\\

The rest of the proof is as in \cite{acs3}, and is as follows. Let
$\mathcal M$ be the space defined in STEP 2, and let $A_{\mathcal
M}$, $C_{\mathcal M}$ denote the matrix representations of $A$ and
$C$, respectively,  in a basis of $\mathcal M$, and let $G_{\mathcal M}$ be the
corresponding Gram matrix. It follows from \eqref{eq:acont} that
\[
A_{\mathcal M}^*G_{\mathcal M}A_{\mathcal M}\le G_{\mathcal
M}-C_{\mathcal M}^*C_{\mathcal M}.
\]

STEP 3: {\sl The equation
\[
A_{\mathcal M}^*P_{\mathcal M}A_{\mathcal M}= P_{\mathcal
M}-C_{\mathcal M}^*C_{\mathcal M}
\]
has a unique solution. It is strictly negative and $\mathcal M$
endowed with this metric is contractively included in $\mathcal
P(S)$.}\smallskip

Recall that the S-spectrum of an operator $A$ is defined as the set of quaternions $p$ such that $A^2-2{\rm Re}(p)A+|p|^2I$ is not invertible, see \cite{MR2752913}. Then, the first two claims follow from the fact that the S-spectrum of
$A_{\mathcal M}$, which coincides with the right spectrum of $A_{\mathcal M}$,
is outside the closed unit ball. Moreover, the matrix
$G_{\mathcal M}-P_{\mathcal M}$ satisfies
\[
A_{\mathcal M}^*(G_{\mathcal M}-P_{\mathcal M})A_{\mathcal M}\le
G_{\mathcal M}-P_{\mathcal M},
\]
or equivalently (since $A$ is invertible)
\[
G_{\mathcal M}-P_{\mathcal M}\le A_{\mathcal M}^{-*}(G_{\mathcal
M}-P_{\mathcal M})A_{\mathcal M}^{-1}
\]
and so, for every $n\in\mathbb N$,
\begin{equation}
\label{eqpm}
G_{\mathcal M}-P_{\mathcal M}\le (A_{\mathcal
M}^{-*})^n(G_{\mathcal M}-P_{\mathcal M})A_{\mathcal M}^{-n}.
\end{equation}
By the spectral theorem (see \cite[Theorem 3.10, p. 616]
{MR2496568} and \cite[Theorem 4.12.6, p. 155]{MR2752913} for the
spectral radius theorem) we have:
\[
\lim_{n\rightarrow\infty}\|A_{\mathcal M}^{-n}\|^{1/n}=0,
\]
and so $\lim_{n\rightarrow\infty}\|(A_{\mathcal
M}^{-*})^n(P_{\mathcal M}-G_{\mathcal M})A_{\mathcal
M}^{-n}\|=0$. Thus entry-wise
\[
\lim_{n\rightarrow\infty}(A_{\mathcal M}^{-*})^n(P_{\mathcal
M}-G_{\mathcal M})A_{\mathcal M}^{-n}=0
\]
and it follows from \eqref{eqpm} that $G_{\mathcal M}-P_{\mathcal M}\le 0$.\\

By \cite[Proposition 8.8]{acs3}
\[
\mathcal M=\mathcal P(B),
\]
when $\mathcal M$ is endowed with the $P_{\mathcal M}$ metric. Furthermore:\\

STEP 4: {\sl The kernel $K_S-K_B$ is positive.}\\

Let $k_{\mathcal M}$ denote the reproducing kernel of $\mathcal M$
when endowed with the $\mathcal P(S)$ metric. Then
\[
k_{\mathcal M}(p,q)-K_B(p,q)\ge 0
\]
and
\[
K_S(p,q)-k_{\mathcal M}(p,q)\ge 0.
\]
Moreover
\[
K_S(p,q)-K_B(p,q)=K_S(p,q)-k_{\mathcal M}(p,q)+k_{\mathcal M}(p,q)-K_B(p,q)
\]
and so it is positive definite.\\

\noindent Finally we apply \cite[Proposition 5.1]{acs3} to
\[
K_S(p,q)-K_B(p,q)=B(p)\star\left(I_r-S_0(p)S_0(q)^*\right)\star_rB(q)^*
\]
where $S_0(p)=B(p)^{-\star}\star S(p)$, to conclude that $S_0$ is a
Schur function.
\end{proof}

\section{The case of the half-space}
\setcounter{equation}{0}

Since the map (where $x_0\in\mathbb R_+$)
\[
p\,\,\mapsto\,\, (p-x_0)(p+x_0)^{-1}
\]
sends the right half-space onto the open unit ball, one can
translate the previous results to the case of the half-space
$\mathbb H_+$. In particular the Blaschke-Potapov factors are of
the form
$$
B_{a}(p,P)=I_r+(b_a(p)-1)P
$$
where $P$ is a matrix such that $P^2=P$ and $JP\geq 0$ where, in
general, $J$ is signature matrix, and $a\in\mathbb H^+$. The
factors of the third type are now functions of the form
\[
I_r-ku\star (p+w_0)^{-\star}u^* J
\]
where $u\in\mathbb H^r$ is such that $uJu^*=0$, and $w_0+\overline{w_0}=0$, $k>0$. The various
definitions and considerations on rational $J$-unitary functions
introduced in Section 5 have counterparts here. We will not
explicit them, but restrict ourselves to the case $J_1=I_s$ and
$J_2=I_r$, and  only mention the counterpart of the Krein-Langer
factorization in the half-space setting. We outline the results
and leave the proofs to the reader.\\

In the setting of slice hyperholomorphic functions in
$\mathbb H_+$ the counterpart of the kernel $\sum_{n=0}^\infty
p^n\overline{q}^n$ is
\begin{equation}
\label{kernel}
k(p,q)=(\bar p +\bar q)(|p|^2 +2{\rm Re}(p) \bar q +\bar q^2)^{-1}.
\end{equation}

\begin{defn}
The $\mathbb H^{r\times s}$-valued function $S$ slice
hypermeromorphic in an axially symmetric s-domain $\Omega$ which
intersects the positive real line belongs to the class $\mathscr
S_\kappa(\mathbb H_+)$ if the kernel
\[
K_S(p,q)=I_r k(p,q)-S(p)\star k(p,q)\star_r S(q)^*
\]
has $\kappa$ negative squares in $\Omega$, where $k(p,q)$ is
defined in (\ref{kernel}).
\end{defn}

The following realization theorem has been proved in
\cite[Theorem 6.2]{2013arXiv1310.1035A}.

\begin{thm}
Let $x_0$ be a strictly positive real number. A $\mathbb
H^{r\times s}$-valued function $S$ slice hyperholomorphic in an
axially symmetric s-domain $\Omega$ containing $x_0$ is the
restriction to $\Omega$ of an element of $\mathcal
S_\kappa(\mathbb H_+)$ if and only if it can be written as
\begin{equation}
\label{realS}
\begin{split}
S(p)&=H-(p-x_0)\left(G-(\overline{p}-x_0)(\overline{p}+x_0)^{-1}GA\right)\times \\
&\hspace{10mm}\times \left(\frac{|p-x_0|^2}{|p+x_0|^2}A^2 - 2{\rm
Re}~\left(\frac{p-x_0}{p+x_0}\right)A+I\right)^{-1}F,
\end{split}
\end{equation}
where $A$ is a linear bounded operator in a right-sided
quaternionic Pontryagin space $\Pi_\kappa$ of index $\kappa$,
and, with $B=-(I+x_0A)$, the operator matrix
\[
\begin{pmatrix}
B&F\\ G&H\end{pmatrix}\,\,:\,\, \begin{pmatrix}\Pi_k\\
\mathbb H^s\end{pmatrix}\,\,\,\longrightarrow \,\,\,\,
\begin{pmatrix}\Pi_k\\ \mathbb H^r\end{pmatrix}
\]
is co-isometric. In particular $S$ has a unique slice
hypermeromorphic extension to $\mathbb H_+$. Furthermore, when
the pair $(G,A)$ is observable, the realization is unique up to a
unitary isomorphism of Pontryagin right quaternionic spaces.
\label{thmrealS}
\end{thm}

By an abuse of notation, we write
\[
S(p)=H-(p-x_0) G\star((x_0+p)I+(p-x_0)B)^{-\star} F
\]
rather than \eqref{realS}.\\

In the following statement, the degree of the Blaschke product
$B_0$ is the dimension of the associated reproducing kernel
Hilbert space with reproducing kernel $K_{B_0}$.

\begin{thm}
Let $S$ be a $\mathbb H^{r\times s}$-valued function slice
hypermeromorphic in an axially symmetric s-domain $\Omega$ which
intersects the positive real line. Then, $S\in\mathscr
S_\kappa(\mathbb H_+)$ if and only if it can be written as
$S=B_0^{-\star}\star S_0$, where $B_0$ is a $\mathbb H^{r\times
r}$-valued finite Blaschke product of degree $\kappa$, and $S_0\in
\mathscr S_0(\mathbb H_+)$.
\end{thm}

\end{document}